\documentclass[11pt,letterpaper]{amsart}
\usepackage{amsmath,amsthm,amssymb,amsfonts,amsaddr}
\usepackage{color}
\usepackage{enumerate}
\usepackage{hyperref}
\usepackage{mathrsfs}
\usepackage{mathtools}
\usepackage{cleveref}
\usepackage{comment}

\theoremstyle{plain}
\newtheorem{theorem}{Theorem}
\newtheorem{proposition}[theorem]{Proposition}
\newtheorem{definition}[theorem]{Definition}

\newtheorem{lemma}[theorem]{Lemma}

\newtheorem{remark}[theorem]{Remark}

\def\Hs{\mathscr{H}}
\def\I{\mathcal{I}}
\def\mf{\mathfrak{m}}
\def\NN{\mathbb{N}}

\def\PP{\mathbb{P}}

\DeclareMathOperator*{\depth}{depth}

\DeclareMathOperator*{\HF}{HF}

\DeclareMathOperator*{\Proj}{Proj}
\DeclareMathOperator*{\reg}{reg}
\DeclareMathOperator*{\Res}{Res}
\DeclareMathOperator*{\sat}{sat}

\DeclareMathOperator*{\Tor}{Tor}
\DeclareMathOperator*{\Tr}{Tr}

\title{On the Castelnuovo-Mumford regularity of subspace arrangements}

\thanks{This work was supported by the National Key R\&D Program of China under Grant No. 2023YFA1009402}

\author{Aldo Conca}
\address{Dipartimento di Matematica,  Universit\`a di Genova, \\Via Dodecaneso 35, 16146 Genova, Italy}
\email{conca@dima.unige.it}

\author{Manolis C. Tsakiris}
\address{State Key Laboratory of Mathematical Sciences, \\ Academy of Mathematics and Systems Science, \\Chinese Academy of Sciences, 100190, Beijing, China}
\email{manolis@amss.ac.cn}

\begin{document}



\maketitle

\vspace{0.3in}

\begin{abstract}
Let $X$ be a union of $n$ generic linear subspaces of codimension $>1$ in $\mathbb{P}^d$. 
Improving an earlier bound due to Derksen and Sidman \cite{derksen2002sharp} we prove that the Castelnuovo-Mumford regularity  of $X$ satisfies $ \reg(X) \le n - \lfloor{n / (2d-1)\rfloor}$.

\end{abstract}

\section{Introduction}

The notion of Castelnuovo-Mumford regularity was introduced by Mumford in \cite{mumford1966lectures}, to systematically control the vanishing of sheaf cohomology. Since then it has proved to be a fundamental notion throughout algebraic geometry and commutative algebra, both from the theoretical and computational perspectives. For instance, proofs for the existence of the Hilbert scheme rely on it \cite{kollar1999rational,mumford1966lectures}, while so do complexity bounds for Gr\"obner bases computations \cite{bayer1993can}. It even plays a role in the theory of subspace clustering \cite{tsakiris2017filtrated} in machine learning. 

In the present work we are concerned with the Castelnuovo-Mumford regularity $\reg (X)$ of the reduced union $X$ of $n$ generic linear spaces $X_1,\dots,X_n$ of arbitrary dimensions in a projective space $\PP_k^d$ over an infinite field $k$. When there are no pairwise intersections between the $X_i$'s, the regularity can be extracted from another important invariant; the Hilbert function. For points, i.e. when $\dim X_i=0$, the Hilbert function is easy to compute; e.g. see \cite{geramita1983hilbert}. However, for higher dimensions determining the Hilbert function is very challenging. A landmark achievement regarding this problem is the work of Hartshorne \& Hirschowitz \cite{hartshorne1982droites} who settled the case of lines. Another important step forward was taken by Derksen \cite{derksen2007hilbert}, who gave a formula for the Hilbert polynomial of $X$ and proved that it agrees with the Hilbert function at degrees $\ge n$. On the other hand, the Hilbert function is largely unknown at degrees $<n$. 

The hardest part of the proof of Hartshorne \& Hirschowitz \cite{hartshorne1982droites} was to treat the case $d=3$, which was done by degeneration techniques via smooth quadric surfaces. Recently a new proof was given by Aladpoosh \& Catalisano \cite{aladpoosh2021hartshorne}, where degeneration techniques via linear spaces replaced those via quadrics. These have the potential of generalizing to higher dimensions and have inspired us in this article. 

For generic lines the formula for the regularity has been derived by Rice \cite{rice2022generic} from \cite{hartshorne1982droites}. Apart from this, there has been a single general result regarding $\reg (X)$, due to Derksen and Sidman \cite{derksen2002sharp}, who proved that $\reg (X) \le n$ for any reduced union of linear spaces, and provided examples of (special) configurations showing that their result is sharp; for special subspace arrangements there exist various results such as \cite{benedetti2018regularity} and \cite{bjorner2005subspace}. 

In this paper we prove:

\begin{theorem} \label{thm:main}
For $n$ generic linear spaces $X_1,\dots,X_n$ of codimension bigger than $1$ in $\mathbb{P}_k^d$, the regularity of their reduced union $X$ satisfies  
$$\reg (X) \le n - \lfloor{n / (2d-1)\rfloor}.$$
\end{theorem}

At the heart of our proof lies a surprising result: 

\begin{theorem} \label{thm:codim2}
For $n=2d-1$ generic linear spaces $X_1,\dots,X_n$ of codimension $2$ in $\mathbb{P}_k^d$, the regularity of their reduced union $X$ is equal to $n-1$ and the saturated ideal $I_X$ that defines $X$ has a linear graded minimal free resolution. 
\end{theorem}

Another device that plays a role in the proof of Theorem \ref{thm:main}, and of potential further significance to the problem of determining $\reg(X)$, is:


\begin{theorem} \label{thm:JL}  
Let $J$ be any non-zero saturated homogeneous ideal of a polynomial ring $S$ of dimension $r$ over $k$. For $c=1,\dots,r$ let $L_{c}$ be an  ideal of $S$ generated by $c$ generic linear forms. Then 
$$ \reg(J\cap L_c) = \left\{
\begin{matrix}
&\reg(J)+1, \, \, & \mbox{ if } c\leq \gamma \\
&\reg(J), \, \, &   \mbox{ if } c> \gamma 
\end{matrix} \right.,$$ where $\gamma = \max\{ c: \reg(J+L_{c})=\reg(J)\}$. 
\end{theorem}

\noindent In Theorem \ref{thm:JL} one always has $\gamma \ge \depth (S/J)$ and in many cases the inequality is strict, e.g. if $J$  has a linear resolution and $c<r$, then $(J+L_c)/L_c$ has a linear resolution as well and thus $\gamma\geq r-1$. 

The bound of Theorem \ref{thm:main} is in general not sharp. As asserted though by Theorem \ref{thm:codim2}, it is sharp when $X$ is the union of $n=2d-1$ codimension-$2$ generic linear spaces in $\mathbb{P}_k^d$. For that particular configuration, Theorem \ref{thm:codim2} implies that the Hilbert function of $X$ is fully determined: it coincides with the Hilbert function of $\mathbb{P}_k^d$ at degrees $< 2d-2$, while it agrees thereafter with the Hilbert polynomial of $X$(described in \cite{derksen2007hilbert}). Moreover, the ideal of $X$ has a linear resolution; we are not aware of other codimension-$2$ configurations where this phenomenon occurs. It further follows from our arguments that whenever the union of $n_0$ codimension-$2$ generic linear spaces in $\mathbb{P}_k^d$ has regularity $n_0- \alpha$, then a union of $n$ generic linear spaces of codimension $\ge 2$ in $\mathbb{P}_k^d$ has regularity $\le n - \alpha \lfloor{n / n_0\rfloor}$.

In \S \ref{section:conventions} we set up notational conventions and in \S \ref{section:preliminaries} we review preliminaries. In \S \ref{section:proof-Thm-codim2} we prove Theorem \ref{thm:codim2} and in \S \ref{section:proof-Thm-JL} we prove Theorem \ref{thm:JL}. Finally we prove Theorem \ref{thm:main} in \S \ref{section:proof-Thm-main}. 

\section{Conventions} \label{section:conventions} 
In the rest of the paper we set $r = d+1$ and we use $\mathbb{P}^{r-1}_k = \Proj S$ for $\mathbb{P}^d_k$, where $S=k[x_1,\dots,x_r]$ is the polynomial ring of dimension $r \ge 3$ over the infinite field $k$, endowed with the standard $\mathbb{Z}$-grading. We denote by $\mathfrak{m}=(x_1,\dots,x_r)$ the unique maximal homogeneous ideal of $S$. For a finitely generated graded $S$-module $M$ we denote by $[M]_{\nu}$ the component of $M$ of degree $\nu$, while the Hilbert function $\HF(M,\nu)$ gives the dimension of the $k$-vector space $[M]_\nu$. 

For $Y$ a closed subscheme of $\PP_k^{r-1}$, we let $\I_Y$ be the ideal sheaf of $Y$, and $I_Y$ the unique graded saturated ideal of $S$ for which $\I_Y = \tilde{I}_Y$ \textemdash {here $\tilde{I}_Y$ is the coherent sheaf induced by $I_Y$}. When we say that a scheme $Y$ contains a scheme $Z$, we will always mean an inclusion of ideals $I_Y \subset I_Z$. Similarly, by the union $Y \cup Z$ of two schemes we will mean the scheme defined by the ideal $I_Y \cap I_Z$. 

With $b$ a non-negative integer and $n$ any integer, we denote the ordinary numerical binomial coefficient as ${n \choose b} = n! / ((n-b)!b !)$, with the convention that it is $0$ if $n < b$. With $t$ an element in a commutative ring that contains $\mathbb{Q}$, we denote the binomial polynomial coefficient as $$\left[ \begin{matrix} t \\ b \end{matrix} \right] = \frac{(t-b+1)\cdots (t-1) t }{ b!}.$$ For any integer $a$ we have $\left[ \begin{smallmatrix} a \\ b \end{smallmatrix} \right] = {a \choose b}$ if and only if $a \ge 0$. 

Finally, for a positive integer $\ell$ we set $[\ell]=\{1,\dots,\ell\}$. 





\section{Preliminaries} \label{section:preliminaries}

\subsection{Regularity}
The Castelnuovo-Mumford regularity of a closed scheme $Y \subset \PP_k^{r-1}$ can be defined as the Castelunuovo-Mumford regularity $\reg(I_Y)$ of the saturated ideal $I_Y \subset S$ that defines $Y$. Here more generally we review the Castelnuovo-Mumford regularity of a finitely generated graded module $M$ over the polynomial ring $S$ \cite{Eisenbud:CA,bruns1998cohen,peeva2010graded}, and quote basic properties that we will need. 

In terms of local cohomology, the regularity of $M$ is defined as 
$$\reg(M) = \max\{i+j: \, [H_{\mathfrak{m}}^i(M)]_j \neq 0\}.$$ In terms of the minimal graded free resolution of $M$, $\reg(M)$ is the maximum among the numbers $b_i - i$, where $b_i$ is the largest degree of a minimal generator of the $i$-th syzygy module of $M$; that is $$\reg(M) = \max\{b - i: \, [{\Tor}_i^S(M,k)]_b \neq 0\}.$$

The following fact is known as the \emph{regularity lemma}:

\begin{proposition} \label{prp:regularity lemma}
Let $0 \rightarrow M' \rightarrow M \rightarrow M'' \rightarrow 0$ be a short exact sequence of finitely generated graded $S$-modules and maps of degree zero. Then 
\begin{align*}
\reg(M) &\le \max \{ \reg(M'), \reg(M'') \} \\
\reg(M') &\le \max \{ \reg(M), \reg(M'')+1 \} \\
\reg(M'') &\le \max \{ \reg(M), \reg(M')-1 \}.
\end{align*} Moreover, if $M'$ has finite length, then $$\reg(M) = \max \{ \reg(M'), \reg(M'') \}.$$ 
\end{proposition}

A linear form $x \in [S]_1$ is called almost regular on $M$, if it is a non-zero divisor on $M / H_{\mathfrak{m}}^0(M)$. When $M = S/J$ for some homogeneous ideal $J$, this condition translates to $x$ being regular on $S/ J^{\sat}$; here $J^{\sat} = J : \mathfrak{m}^{\infty}$ is the saturation of $J$ with respect to $\mathfrak{m}$. We recall also that the saturation index $\sat(J)$ of $J$ is the smallest integer $s$ for which $[J^{\sat}]_\nu = [J]_\nu$ for every $\nu \ge s$, i.e. $\sat(J) = \reg(J^{\sat}/J)+1$. With that, the behavior of the regularity upon taking quotients by almost regular elements is extremely useful for inductive arguments on the dimension:

\begin{proposition}\label{prp:almost regular}
Let $J$ be a homogeneous ideal of $S$ and $x \in S$ a linear form almost regular on $S/J$. Then $$\reg(J) = \max\{\reg(J+(x)), \, \sat(J) \},$$ or equivalently 
$$\reg(S/J) = \max\{\reg(S/(J+(x))), \, \reg(J^{\sat}/J) \}.$$
\end{proposition} \noindent For a proof of the above folklore fact we refer to \cite{conca2003castelnuovo}.

\subsection{Subspace Arrangements}

Given an arrangement of $n$ $k$-vector subspaces $V_i, \, i \in [n]$, of $[S]_1$, and the ideals $I_i$ that the $V_i$'s generate, we may consider the intersection $\cap_{i \in [n]} I_i$ or the product  $\prod_{i \in [n]} I_i$ of the $I_i$'s. The intersection ideal is more geometric, because it is the radical ideal that defines the schematic union $X$ of the corresponding arrangement of linear spaces $X_i = \Proj (S/I_i)$ of $\PP_k^{r-1}$. But it is a much harder object to understand than the product ideal $\prod_{i \in [n]} I_i$; already describing a set of generators of $\cap_{i \in [n]} I_i$ is a difficult task \cite{ida1990homogeneous}. On the other hand, many properties of $\prod_{i \in [n]} I_i$ are by now well understood, including a complete description of its minimal graded free resolution and of a minimal primary decomposition \cite{conca2019resolution}. We will keep the notation of this paragraph for the rest of the paper.

\begin{definition} \label{dfn:linearly-general}
We call the subspace arrangement $V_1,\dots,V_n$ of $[S]_1$ linearly general, if for every $A \subseteq [n]$ we have
$$\dim_k \sum_{i \in A} V_i = \min\left\{r, \, {\sum}_{i \in A} \dim_k V_i\right\}.$$ We call the subspace arrangement $X_1, \dots, X_n$ of $\mathbb{P}_k^{r-1}$ linearly general, if the $V_i$'s are linearly general.
\end{definition}

The following result allows valuable access from the product ideal to the intersection ideal. 

\begin{proposition}[Conca \& Herzog, \cite{conca2003castelnuovo}] \label{prp:Conca-Herzog}
Suppose that the subspace arrangement $V_i, \, i \in [n]$, of $[S]_1$ is linearly general. Then for every $\nu \ge n$ we have $$\textstyle [\cap_{i \in [n]} I_i]_\nu = \big[\prod_{i \in [n]} I_i\big]_\nu.$$
\end{proposition}

Using a formula for the primary decomposition that they developed, 
Conca \& Herzog \cite{conca2003castelnuovo}  proved that $$\textstyle \reg \big(\prod_{i \in [n]} I_i \big)=n,$$ while Derksen \& Sidman proved:

\begin{proposition}[Derksen \& Sidman, \cite{derksen2002sharp}] \label{prp:Derksen-Sidman}
 $$\reg ( \cap_{i \in [n]} I_i) \le n.$$
\end{proposition}

Later on, Derksen gave a formula for the Hilbert polynomial and proved that it agrees with the Hilbert function from degrees $n$ and on:

\begin{proposition}[Derksen, \cite{derksen2007hilbert}] \label{prp:Derksen}
Suppose that the subspace arrangement $V_i, \, i \in [n]$, of $[S]_1$ is linearly general. For any subset $A \subseteq [n]$ set $d_A = \dim_k \sum_{i \in A} V_i$. Then for every $\nu \ge n$ we have 
$$\HF\big(S/\cap_{i \in [n]} I_i,\nu \big) = \sum_{\emptyset \neq A \subseteq [n]} (-1)^{|A|+1} {\nu+r-1-d_A \choose r-1-d_A},$$ where all terms with $r-1-d_A<0$ are zero by convention.
\end{proposition} 

\begin{remark}
The notion of a linearly general subspace arrangement in projective space given in Definition \ref{dfn:linearly-general} is equivalent to the notion of a transversal subspace arrangement as defined in \cite{derksen2007hilbert}.
\end{remark}


\section{Proof of Theorem \ref{thm:codim2}} \label{section:proof-Thm-codim2}


\subsection{First Hilbert Function Lemma} \label{section:1st-HF}

We begin by giving a formula for the Hilbert function at degree $\ell-1$ of a linearly general subspace arrangement of $\ell$ codimension-$2$ linear spaces in $\mathbb{P}_k^{r-1}$.  

\begin{lemma} \label{lem:HF at n-1}
Suppose that the subspace arrangement $V_i, \, i \in [\ell]$, of $[S]_1$ is linearly general and $\dim_k V_i=2$ for every $i \in [\ell]$. Then 
$$\HF(\cap_{i \in [\ell]} I_i, \ell-1) = \sum_{j \ge 0}(-1)^{j}  {\ell \choose j}   {\ell+r-2-2j \choose \ell-1}.$$
\end{lemma}
\begin{proof}
Consider the short exact sequence
$$0 \rightarrow \frac{S}{\cap_{i \in [\ell]}I_i} \rightarrow \frac{S}{\cap_{i \in [\ell-1]}I_i} \oplus \frac{S}{I_{\ell}} \rightarrow \frac{S}{I_{\ell}+\cap_{i \in [\ell-1]}I_i} \rightarrow 0.$$ By Proposition \ref{prp:Derksen} 
\begin{align*}
&\HF\left(\frac{S}{\cap_{i \in [\ell-1]}I_i},\ell-1 \right)= \\
&\sum_{\emptyset \neq A \subseteq [\ell-1]} (-1)^{|A|+1} {\ell-1+r-1-2|A| \choose r-1 -2|A|} \\
&=  \sum_{j\ge 1} (-1)^{j+1} {\ell-1 \choose j} {\ell+r-2-2j \choose r-1 -2j}  \nonumber.
\end{align*} 

\noindent By Proposition \ref{prp:Conca-Herzog} $$ \left[\bigcap_{i \in [\ell-1]}I_i\right]_{\ell-1} = \left[\prod_{i \in [\ell-1]}I_i\right]_{\ell-1}.$$ Thus, with $\bar{S} = S/I_{\ell}$ and $\bar{I}_i = (I_i + I_{\ell}) / I_{\ell}$, we have $$\left[\frac{S}{I_{\ell}+\bigcap_{i \in [\ell-1]}I_i}\right]_{\ell-1} = \left[\frac{\bar{S}}{ {\prod}_{i \in [\ell-1]}\bar{I}_i}\right]_{\ell-1}.$$ Since the ideals $\bar{I}_i, \, i \in [\ell-1]$, are generated by a linearly general subspace arrangement of $\ell-1$ dimension-$2$ linear subspaces of $[\bar{S}]_1$, where $\bar{S}$ is a polynomial ring over $k$ in $r-2$ variables, we can compute the Hilbert function of the last term of the short exact sequence at degree $\ell-1$ again by Proposition \ref{prp:Derksen}: 
\begin{align*}
&\HF\left(\frac{\bar{S}}{\prod_{i \in [\ell-1]}\bar{I}_i},\ell-1 \right)= \\
& \sum_{\emptyset \neq A \subseteq [r-2]} (-1)^{|A|+1} {\ell-1+r-3-2|A| \choose r-3 -2|A|} \\
&=  \sum_{j \ge 1} (-1)^{j+1} {\ell-1 \choose j} {\ell+r-4-2j \choose r-3 -2j}  .
\end{align*} Putting everything together, we arrive at
\begin{align*}
&\HF(\cap_{i \in [\ell]} I_i,\ell-1 ) =\HF(\cap_{i \in [\ell-1]} I_i,\ell-1 ) - \HF(\textstyle{\prod_{i \in [\ell-1]}} \bar{I}_i,\ell-1)  \\
& =   \sum_{j \ge 0}(-1)^{j} {\ell-1 \choose j} {\ell+r-2-2j \choose r-1 -2j} \\
&- \sum_{j \ge 0} (-1)^{j} {\ell-1 \choose j} {\ell+r-4-2j \choose r-3 -2j}  \\
& =  \sum_{j \ge 0} (-1)^{j} {\ell-1 \choose j} {\ell+r-2-2j \choose r-1 -2j}  \\
&+ \sum_{j \ge 1} (-1)^{j} {\ell-1 \choose j-1} {\ell+r-2-2j \choose r-1 -2j}  \\
& =  \sum_{j \ge 0} (-1)^{j} \bigg[ {\ell-1 \choose j} +{\ell-1 \choose j-1}\bigg]  {\ell+r-2-2j \choose r-1 -2j}  \\
& =  \sum_{j \ge 0} (-1)^{j}  {\ell \choose j}   {\ell+r-2-2j \choose r-1 -2j}.  \qed 
\end{align*} \phantom\qedhere
\end{proof}

\subsection{Second Hilbert Function Lemma}

We recall a well-known binomial identity: 

\begin{lemma} \label{lem:binomial}
Let $p(t)$ be a polynomial in $t$ and $d>\deg(p)$. Then 
$$\sum_{j=0}^d (-1)^j {d \choose j} p(j)=0.$$
\end{lemma}

We have the following interesting fact. 

\begin{lemma} \label{lem:HF=1}
Through the $r-1$ linearly general codimension-$2$ linear spaces $X_i, \, i\in[r-1]$, of $\PP_k^{r-1}$ passes a unique hypersurface of degree $r-2$, i.e. 
$$\HF(\cap_{i \in [r-1]} I_i, r-2) = 1.$$
\end{lemma}
\begin{proof} 
Lemma \ref{lem:HF at n-1} gives us
\begin{align*}
&\HF(\cap_{i \in [r-1]} I_i, r-2)= \\
& =  \sum_{j \ge 0} (-1)^{j}  {r-1 \choose j}   {2r-3-2j \choose r-2} 
\end{align*}
\begin{align*}
&= \sum_{j = 0}^{\lfloor (2r-3)/2 \rfloor } (-1)^{j}  {r- 1\choose j} \begin{bmatrix} 2r-3-2j \\ r-2 \end{bmatrix} \\
&= \sum_{j = 0}^{r-2} (-1)^{j}  {r- 1\choose j} \begin{bmatrix} 2r-3-2j \\ r-2 \end{bmatrix}.
\end{align*} As a polynomial in $t$, the degree of $$p(t)=\begin{bmatrix} 2r-3-2t \\ r-2 \end{bmatrix}$$ is $r-2$. Hence Lemma \ref{lem:binomial} gives
$$\sum_{j=0}^{r-1} (-1)^{j}  {r- 1\choose j} \begin{bmatrix} 2r-3-2j \\ r-2 \end{bmatrix} = 0. $$ Now we are done by observing that the last term in the summation above for $j=r-1$ is equal to $-1$. \end{proof}

\subsection{Sundials}
The notion of a \emph{sundial} played a crucial role in the proof of \cite{hartshorne1982droites}. It was also used in \cite{hirschowitz1981postulation} and has more recently been generalized in various ways; e.g. see \cite{aladpoosh2021hartshorne}. 
Here we discuss yet another generalization that we will employ in the proof of Lemma \ref{lem:zero in degree n-2}.

\begin{definition} \label{dfn:sundial}
In $\PP^{r-1}_k$ consider two linear spaces $Y$ and $Z$, of codimensions $2$ and $i+2$ respectively, which lie in a hyperplane $H$, they intersect at a codimension $i+3$ linear space $W$, and $0 \le i \le r-4$. A $(2,i+2)$-sundial is a scheme defined by an ideal sheaf of the form $\I_Y \cap \I_Z \cap \I_W^2$. We say that the hyperplane $H$ supports the sundial.
\end{definition}

\begin{lemma} \label{lem:sundial}
Let $Y,Z$ be linearly general linear spaces in $\PP^{r-1}_k$ of codimensions $2$ and $2+i$ respectively, such that $4+i \le r$. Then there exists a flat family $E_c, \, c \in k$ of closed subschemes of $\PP^{r-1}_{k}$, such that $E_{c \neq 0}$ is isomorphic to $Y \cup Z$, while $E_0$ is a $(2,i+2)$-sundial. 
\end{lemma}
\begin{proof}
We may assume that $I_Y=(x_1,x_2)$ and $I_Z=(x_3,\dots,x_{i+4})$. With $\lambda$ a new indeterminate of degree $0$, we define in the extended polynomial ring $S[\lambda]$ the ideal 
$$J_\lambda = \big(x_1,\lambda(x_2-x_3)+x_3\big) \cap (x_3,\dots,x_{i+4}).$$ Since $x_1,\lambda(x_2-x_3)+x_3$ is a regular sequence in $S[\lambda]$ and $S[\lambda] / (x_3,\dots,x_r)$, we have 
$$J_\lambda = \big(x_1,\lambda(x_2-x_3)+x_3\big) (x_3,\dots,x_{i+4}).$$ Moreover, the canonical ring homomorphism $\varphi: k[\lambda] \rightarrow S[\lambda] / J_\lambda$ is flat if and only if $S[\lambda]/J_\lambda$ is torsion-free as a $k[\lambda]$-module \textemdash{now this follows immediately once one notes that $\big(x_1,\lambda(x_2-x_3)+x_3\big)$ and  $(x_3,\dots,x_{i+4})$ are prime ideals of $S[\lambda]$, whose contraction in $k[\lambda]$ is zero}.

The ideal $J_\lambda$ induces a flat family of closed subschemes $E_c, \, c \in k$ of $\PP_k^{r-1}$. The ideal of $S$ defining $E_c$ is $$J_c = \big(x_1,c(x_2-x_3)+x_3\big) (x_3,\dots,x_{i+4}).$$ For $c \neq 0$ we have that 
$$J_c = \big(x_1,c(x_2-x_3)+x_3\big) \cap (x_3,\dots,x_{i+4}),$$ defines two codimension $2$ and $i+2$ linear spaces of $\PP_k^{r-1}$, which intersect at codimension $i+4$. On the other hand,
\begin{align}
 J_0 & = \big(x_1,x_3\big) (x_3,\dots,x_{i+4}) \nonumber \\
 & = (x_1,x_3) \cap (x_3,\dots,x_{i+4}) \cap (x_1,x_3,\dots,x_{i+4})^2 \nonumber 
 \end{align} defines a $(2,i+2)$-sundial.
\end{proof}

\subsection{A Variant of Castelnuovo's Inequality} \label{section:Castelnuovo's inequality}

For a homogeneous ideal $J$ of $S$ and any linear form $x$, the short exact sequences
$$0 \rightarrow S/(J:(x)) (-1) \stackrel{x}{\rightarrow} S/J \rightarrow S/ (J+(x)) \rightarrow 0,$$
$$0 \rightarrow S(-1) \stackrel{x}{\rightarrow} S \rightarrow S/(x) \rightarrow 0,$$ give
$$\HF(J,\nu) = \HF\big((J+(x))/(x),\nu\big) + \HF(J:(x),\nu-1).$$ Castelnuovo's inequality \textemdash{usually quoted in sheaf-theoretic form\textemdash} is an easy consequence of the equality above: $$\HF(J,\nu) \le \HF\big(\big((J+(x))/(x)\big)^{\sat},\nu\big) + \HF(J:(x),\nu-1).$$

Call $Y$ the closed subscheme of $\PP_k^{r-1}$ defined by $J$ and $H$ the hyperplane defined by the linear form $x$. Then the closed subscheme of $\PP_k^{r-1}$ defined by the ideal $J:(x)$ is called the residual scheme of $Y$ with respect to $H$, denoted as $\Res_H(Y)$. The closed subscheme of $\PP_k^{r-2}$ defined by the ideal $(J+(x))/(x)$ is called the trace of $Y$ on $H$, and is denoted by $\Tr_H(Y)$ \textemdash note $((J+(x))/(x))^{\sat}$ also defines $\Tr_H(Y)$. 

In the proof of Lemma \ref{lem:zero in degree n-2} we will be concerned with showing $\HF(J,\nu)=0$ for certain ideals $J$ and degrees $\nu$. Often, the scheme $Y$ defined by the ideal $J$ will be the reduced union in $\PP_k^{r-1}$ of linear spaces $Y_i$, together with a $(2,j+\ell)$-sundial $U$, for some $j, \ell \ge 0$. That is, $J = (\cap_i I_{Y_i}) \cap I_U$. The linear form $x$ will always be chosen to define a hyperplane that supports the sundial; e.g., in the notation of the proof of Lemma \ref{lem:sundial}, $I_U=J_0$ and we would take $x=x_3$. Instead of dealing with the trace scheme $\Tr_H(Y)$ defined by the possibly non-saturated ideal of $S/(x)$, $$\frac{(\cap_i I_{Y_i}) \cap I_U+(x)}{(x)},$$ it will be more convenient to work with its subscheme $\Tr^\dagger_H(Y)$, defined by the larger ideal 
$$\left(\bigcap_i \frac{I_{Y_i} +(x)}{(x)}\right) \cap \left(\frac{I_U+(x)}{(x)}\right).$$ Since the sundial will always be of type $(2,j+\ell)$ and $(x)$ defines a hyperplane that supports the sundial, we see that $\Tr_H(U)$ \textemdash{ defined by the ideal $(I_U+(x))/(x)$\textemdash} is the union in $\PP_k^{r-2}$ of a hyperplane together with a codimension-$(j+\ell-1)$ linear space $Z$. Thus $\Tr_H(U)$ always contributes a hyperplane as an irreducible component of $\Tr^\dagger_H(Y)$, and possibly one more whenever $j+\ell-1=1$. Let $1 \le \alpha \le 2$ be the number of the hyperplanes that appear as irreducible components of $\Tr_H(U)$. Denote by $\Tr^{\dagger \dagger}_H(Y)$ the scheme obtained by removing these components from $\Tr^{\dagger}_H(Y)$. Using Castelnuovo's inequality, it becomes clear that to show no  hypersurface of $\PP_k^{r-1}$ of degree $\nu$ contains $Y$, it suffices to show that i) there is no hypersurface of $\PP_k^{r-1}$ of degree $\nu-1$ that contains $\Res_H(Y)$, and ii) there is no hypersurface of $\PP_k^{r-2}$ of degree $(\nu-\alpha)$ that contains $\Tr^{\dagger \dagger}_H(Y)$. The arguments in Section \ref{section:2nd-HF} make repeated use of this idea. 

\subsection{Third Hilbert Function Lemma} \label{section:2nd-HF}
The hardest part of proving Theorem \ref{thm:codim2} is establishing the following key fact:

\begin{lemma} \label{lem:zero in degree n-2}
Under the assumption of Theorem \ref{thm:codim2} with $n=2r-3$, we have $$\HF(\cap_{i \in [n]} I_i, n-2) = 0.$$
\end{lemma}


From the upper-semicontinuity of the Hilbert function in a flat family, it suffices to show that this holds for a special member of a flat family of schemes, whose general member is the union of $n$ codimension-$2$ generic linear spaces of $\PP_k^{r-1}$. We will be denoting the statement of Lemma \ref{lem:zero in degree n-2} by $$\Hs_{2r-3,\, 2r-5,\, r}$$ \textemdash the first index refers to the number of linear spaces in the arrangement, the second index refers to the degree of interest of the ideal of the arrangement, and the last index is the number of ambient variables. 

To prove Lemma \ref{lem:zero in degree n-2}, we proceed by induction on $r$. As a base for the induction, we take $r=3$ \textemdash the statement is then trivial, since no line in $\PP^2$ contains three generic points. In the sequel we assume that i) $r\ge 4$, and ii) $\Hs_{2r'-3,2r'-5,r'}$ is true for any $3 \le r' <r$. 

Within this induction hypothesis, we prove an auxiliary statement:

\begin{lemma} \label{lem:auxiliary}
Let $0 \le j \le r-4$. Consider a union in $\PP_k^{r-1}$ of $2r-5-j$ codimension-$2$ generic  linear spaces, together with a codimension-$(j+3)$ generic linear space. Then there is no hypersurface of degree $2r-6-j$ containing this union. 
\end{lemma}

We denote the statement of Lemma \ref{lem:auxiliary} by $$\Hs'_{2r-5-j, \, j+3, \, 2r-6-j, \, r},$$ where the first
index refers to the number of codimension-$2$ linear spaces in the arrangement, the second index refers to the codimension of the remainder single linear space in the arrangement, the third index refers to the degree of interest of the ideal of the arrangement, and the last index is the number of ambient variables.

We prove the statement by ascending induction on $r$ and descending induction on $j$. For the base of our induction on $r$ we take $r=4$ \textemdash{the only possibility is $j=0$ and the statement follows from the well-known fact that there is a unique quadric surface in $\PP^3$ through three generic lines, and hence it suffices to take the point outside the quadric (thus no quadric contains three generic lines and a generic point)}. For any $r$, the base for the induction on $j$ is for $j = r-4$. The statement then becomes $\Hs'_{r-1,r-1,r-2,r}$. This follows immediately from Lemma \ref{lem:HF=1}, because through $r-1$ generic codimension-$2$ linear spaces in $\PP_k^{r-1}$ passes a unique hypersurface of degree $r-2$, and the codimension-$(r-1)$ space is a generic point, which can be taken to be outside that hypersurface. In the sequel we will assume that i) $\Hs'_{2r'-5-j, \, j+3, \, 2r'-6-j, \, r'}$ is true for any $r' < r$ and any $j$ such that $0 \le j \le r'-4$, and ii) $\Hs'_{2r-5-j', \, j'+3, \, 2r-6-j', \, r}$ is true for any $r$ and any $j'$ such that $ j<j'\le r-4$.

With $0 \le j<r-4$, let $Y$ be the scheme of Lemma \ref{lem:auxiliary}. Since $2+(j+3)\le r$, we apply Lemma \ref{lem:sundial} to degenerate $Y$ into a scheme $Y'$ that consists of $2r-6-j$ codimension-$2$ generic linear spaces and a $(2,j+3)$-sundial. Now we apply the variant of Castelnuovo's inequality described in Section \ref{section:Castelnuovo's inequality} \textemdash it is enough to show that i) there is no hypersurface of $\PP_k^{r-1}$ of degree $2r-7-j$ that contains the residual scheme $\Res_H(Y')$ of $Y'$ with respect to a hyperplane $H$ that supports the sundial, and ii) there is no hypersurface of $\PP_k^{r-2}$ of degree $2r-7-j$ that contains $\Tr^{\dagger\dagger}_H(Y')$.

The residual scheme $\Res_H(Y')$ is the union in $\PP_k^{r-1}$ of $2r-6-j$ codimension-$2$ generic linear spaces, together with a codimension-$(j+4)$ generic linear space. The fact that such a scheme lies in no hypersurface of degree $2r-7-j$ is the statement  $$\Hs'_{2r-5-(j+1), \, (j+1)+3, \, 2r-6-(j+1), \, r}.$$ As $j+1 \le r-4,$ this is true by our induction hypothesis on $j$ for $\Hs'$. 
 
Next, $\Tr^{\dagger\dagger}_H(Y')$ is the union in $\PP_k^{r-2}$ of $2r-6-j$ codimension-$2$ generic linear spaces, together with a codimension-$(j+2)$ generic linear space. We claim that there is no degree-$(2r-7-j)$ hypersurface of $\PP_k^{r-2}$ that contains $\Tr^{\dagger\dagger}_H(Y')$. If $j=0$, this is the statement $$\Hs_{2(r-1)-3,2(r-1)-5,r-1},$$ which is true by our induction hypothesis on $r$ for $\Hs$. If $j>0$, it is the statement  $$\Hs'_{2(r-1)-5-(j-1), \, (j-1)+3, \, 2(r-1)-6-(j-1), \, r-1},$$ which is true by our induction hypothesis on $r$ for $\Hs'$ (note that $j-1 < (r-1)-4$). This concludes the proof of Lemma \ref{lem:auxiliary}.

We now finish the proof of Lemma \ref{lem:zero in degree n-2}. So let $Y=\bigcup_{i \in [2r-3]} Y_i$ be the union in $\PP_k^{r-1}$of $2r-3$ codimension-$2$ generic linear spaces $Y_i$. Since $r \ge 4$, by Lemma \ref{lem:sundial} we degenerate $Y_{2r-4} \cup Y_{2r-3}$ within a flat family to a $(2,2)$-sundial. The resulting scheme $Y'$ is the union in $\PP_k^{r-1}$ of $2r-5$ codimension-$2$ generic linear spaces, together with a generic $(2,2)$-sundial. By the upper-semicontinuity of the Hilbert function in a flat family, it is enough to show that there is no hypersurface of degree $2r-5$  that contains $Y'$. We do this by applying the variant of Castelnuovo's inequality described in Section \ref{section:Castelnuovo's inequality}. So let $H$ be a hyperplane that supports the $(2,2)$-sundial. The residual scheme $\Res_H(Y')$ is the union in $\PP_k^{r-1}$of $2r-5$ codimension-$2$ generic linear spaces together with a codimension-$3$ generic linear space. Saying that there is no hypersurface of degree $2r-6$  that contains $\Res_H(Y')$, is the same as saying that $$\Hs'_{2r-5,\, 3,\, 2r-6, \, r}$$ holds true \textemdash now this follows from Lemma \ref{lem:auxiliary} with $j=0$. Next, $\Tr^{\dagger \dagger}_H(Y')$ is the union in $\PP_k^{r-2}$ of $2r-5$ codimension-$2$ generic linear spaces. We must show that there is no hypersurface of degree $2r-7$ that contains $\Tr^{\dagger \dagger}_H(Y')$ (note that $\alpha=2$ in the notation of Section \ref{section:Castelnuovo's inequality}). But this is the statement $$\Hs_{2(r-1)-3, \, 2(r-1)-5, \, r-1},$$ which is true by our induction hypothesis on $r$ for $\Hs$. 

\subsection{Saturation} The last ingredient is:
\begin{lemma}\label{lem:saturation}
Let $J_i, \, i \in [\ell]$, be ideals generated by linear forms; suppose $\dim S/J_i \ge 2$ for at least one $i$. Then for a generic linear form $h$ we have
\begin{align*}
\big[\cap_{i\in [\ell]} J_i + (h) \big]^{\sat} = \cap_{i \in [\ell]} \big(J_i + (h)\big). 
\end{align*}
\end{lemma}
\begin{proof}
If some $J_s$ is the maximal homogeneous ideal $\mathfrak{m}$, then all ideals $\cap_{i\in [\ell]} J_i, \, \cap_{i\in [\ell]} J_i + (h), \, \cap_{i \in [\ell]} \big(J_i + (h))$ stay the same after removing it, so that we may assume $J_i \subsetneq \mathfrak{m}$ for every $i \in [\ell]$.

Consider the short exact sequence $$0 \rightarrow \frac{S}{\cap_{i\in [\ell]} J_i}  \stackrel{\varphi}{\rightarrow} \prod_{i \in [\ell]} S/J_i \rightarrow C \rightarrow 0,$$ where $\varphi$ takes the class of $a \in S$ in $S/\cap_{i\in [\ell]} J_i$ to the classes of $a$ mod $J_i$ for $i\in [\ell]$, and $C$ is the cokernel. Tensoring with $S/(h)$ and using the fact that $h$ is $S/J_i$-regular for every $i$, we get 
$$0 \rightarrow {\Tor}_1^S\left(C, \frac{S}{(h)}\right) \rightarrow  \frac{S}{\cap_{i\in [\ell]} J_i+(h)}  \stackrel{\varphi}{\rightarrow} \prod_{i \in [\ell]} \frac{S}{J_i+(h)} \rightarrow \frac{C}{hC} \rightarrow 0.$$ We conclude that
$$M:=\frac{\cap_{i \in [\ell]} \big(J_i + (h)\big)}{\cap_{i\in [\ell]} J_i + (h)} = {\Tor}_1^S\left(C, \frac{S}{(h)}\right).$$ Now, we can consider $M$ as the kernel of the multiplication map $C(-1) \stackrel{h}{\rightarrow} C$. Since $h$ is generic, it is almost regular on $C(-1)$, and so $M$ has finite length. Thus with $A = \cap_{i \in [\ell]} \big(J_i + (h)\big)$ and $B=\cap_{i\in [\ell]} J_i + (h)$ we have $B \subseteq A \subseteq B^{\sat}$. This gives $B^{\sat} = A^{\sat}$. But $A$ is radical because the ideals $J_i+(h)$ are prime and at least one of them is properly contained in the maximal homogeneous ideal of $S$. This implies that $A = A^{\sat}$.
\end{proof}

\subsection{Finishing the Proof of Theorem \ref{thm:codim2}}

We proceed by induction on $r$; the case $r=3$ is a simple exercise, so we assume $r \ge 4$.

By Lemma \ref{lem:zero in degree n-2} the ideal $\cap_{i \in [n]}I_i$ is zero in degree $n-2$ and so necessarily
\begin{align*}
\reg \left( \frac{S}{\cap_{i \in [n]}I_i} \right) \ge n-2.
\end{align*} In what follows we concentrate on proving the reverse inequality.

Let $h \in [S]_1$ be a generic linear form. As such, $h$ is regular on $S / \cap_{i \in [n]}I_i$, and so Proposition \ref{prp:almost regular} gives $$\reg\left( \frac{S}{\cap_{i \in [n]}I_i}\right)=\reg\left( \frac{S}{ \cap_{i \in [n]}I_i + (h)}\right).$$ It thus suffices to show that $$\reg\left( \frac{S}{ \cap_{i \in [n]}I_i + (h)}\right) = n-2.$$ Towards that end, we consider the short exact sequence 
$$ 0 \rightarrow \frac{\cap_{i \in [n]} \big(I_i+(h)\big) } { \cap_{i \in [n]} I_i+(h)} \rightarrow \frac{S}{\cap_{i \in [n]} I_i+(h)} \rightarrow \frac{S}{\cap_{i \in [n]} \big(I_i+(h)\big)} \rightarrow 0.$$ By Lemma \ref{lem:saturation}, the first module in the exact sequence has finite length. Thus Proposition \ref{prp:regularity lemma} gives
$$\reg \left(\frac{S}{\cap_{i \in [n]} I_i+(h)}\right) = $$
 $$\max \left\{\reg\left(\frac{\cap_{i \in [n]} \big(I_i+(h)\big) } { \cap_{i \in [n]} I_i+(h)} \right), \reg\left(\frac{S}{\cap_{i \in [n]} \big(I_i+(h)\big)} \right) \right\}. $$
 
Set $\bar{S} = S/(h)$ and $\bar{I}_i = (I_i + (h)) / (h)$. Then $$\frac{S}{\cap_{i \in [n-2]} (I_i+(h))} = \frac{\bar{S}}{ \cap_{i \in [n-2]} \bar{I}_i},$$ and our induction hypothesis on $r$ gives $$\reg \left(\frac{\bar{S}}{ \cap_{i \in [n-2]} \bar{I}_i} \right) = n-4.$$ Hence $\reg (\cap_{i \in [n-2]} \bar{I}_i) = n-3$. Since taking quotient with a generic linear form can not increase the regularity (Proposition \ref{prp:almost regular}), we have 
$$ \reg (\cap_{i \in [n-2]} \bar{I}_i + \bar{I}_{n-1}) \le n-3.$$ This, together with the short exact sequence 
$$0 \rightarrow \frac{\bar{S}}{\cap_{i \in [n-1]} \bar{I}_i} \rightarrow \frac{\bar{S}}{\cap_{i \in [n-2]} \bar{I}_i} \oplus \frac{\bar{S}}{\bar{I}_{n-1}} \rightarrow \frac{\bar{S}}{\cap_{i \in [n-2]} \bar{I}_i + \bar{I}_{n-1}} \rightarrow 0,$$ and the regularity lemma, give 
$$ \reg(\cap_{i \in [n-1]} \bar{I}_i) \le n-2.$$ One more application of this argument with $\cap_{i \in [n-1]} \bar{I}_i$ and $\bar{I}_n$ gives 
$$ \reg(\cap_{i \in [n]} \bar{I}_i) \le n-1.$$
We have thus shown that $$\reg \left(\frac{S}{ \cap_{i \in [n]} (I_i+(h))} \right) \le n-2,$$ 
and so it remains to prove 
$$\reg \left(\frac{\cap_{i \in [n]} \big(I_i+(h)\big) } { \cap_{i \in [n]} I_i+(h)} \right) \le n-2.$$ 
Since $\reg (\cap_{i \in [n]} I_i+(h)) \le n$ by Propositions \ref{prp:almost regular} and \ref{prp:Derksen-Sidman}, we have that $\cap_{i \in [n]} I_i+(h)$ is saturated from degree $n$ and on. Hence in view of Lemma \ref{lem:saturation}, $\cap_{i \in [n]} \big(I_i+(h)\big)$ and $\cap_{i \in [n]} I_i+(h)$ agree from degree $n$ and on, so it is enough to prove that 
$$ \HF(\cap_{i \in [n]} (I_i+(h)), n-1 ) = \HF (\cap_{i \in [n]} I_i+(h),n-1 ).$$ 

\noindent For the Hilbert function on the left we note
$$ \HF\left(\frac{S}{\cap_{i \in [n]} (I_i+(h))}, n-1 \right) = \HF\left(\frac{\bar{S}}{ \cap_{i \in [n]}\bar{I}_i},n-1\right),$$
and now
 $$\HF(\cap_{i \in [n]} (I_i+(h)), n-1 ) = \HF(S,n-2)+  \HF(\cap_{i \in [n]}\bar{I}_i,n-1).$$\nonumber 
 
\noindent To get a handle on $\HF (\cap_{i \in [n]} I_i+(h),n-1)$, we work with the short exact sequence
$$ 0 \rightarrow \frac{S}{\cap_{i \in [n]}I_i} (-1) \stackrel{h}{\rightarrow} \frac{S}{\cap_{i \in [n]}I_i} \rightarrow \frac{S}{\cap_{i \in [n]} I_i+(h)} \rightarrow 0,$$ 
whose degree $n-1$ component gives 
\begin{align*}
\HF (\cap_{i \in [n]} I_i+(h),n-1 ) =& \HF (\cap_{i \in [n]} I_i,n-1) -\HF (\cap_{i \in [n]} I_i,n-2)\\ & +\HF(S,n-2).
\end{align*} By Lemma \ref{lem:zero in degree n-2},  
$$\HF (\cap_{i \in [n]} I_i,n-2) = 0, $$
and so 
$$\HF (\cap_{i \in [n]} I_i+(h),n-1 ) = \HF(S,n-2)+ \HF (\cap_{i \in [n]} I_i,n-1). $$
To finish the proof we have to show that 
$$ \HF(\cap_{i \in [n]}\bar{I}_i,n-1) = \HF (\cap_{i \in [n]} I_i,n-1). $$ Both these Hilbert function values are accessible via Lemma \ref{lem:HF at n-1}. Indeed, they are equal:

\begin{lemma}
We have that 
$$ \sum_{j \ge 0} (-1)^{j}  {n \choose j}   {n+r-2-2j \choose n-1} = \sum_{j \ge 0} (-1)^{j}  {n \choose j}   {n+r-3-2j \choose n-1}. $$
\end{lemma}
\begin{proof}
The statement is equivalent to 
$$ \sum_{j \ge 0} (-1)^{j}  {n \choose j}   {n+r-3-2j \choose n-2}=0, $$ and recalling that $n=2r-3$, equivalent to
$$ \sum_{j \ge 0} (-1)^{j}  {2r-3 \choose j}   {3r-6-2j \choose 2r-5}=0. $$ The degree of the polynomial $$ \begin{bmatrix}  3r-6-2t \\ 2r-5 \end{bmatrix}$$ is $2r-5$, so Lemma \ref{lem:binomial} gives 
$$ \sum_{j = 0}^{2r-3} (-1)^{j}  {2r-3 \choose j}   \begin{bmatrix}  3r-6-2j \\ 2r-5 \end{bmatrix}=0. $$  In the summation above there are $2r-2$ terms. For $j = 0,\dots,r-2$, we claim that the terms corresponding to $j$ and $2r-3-j$ are equal. To see this, with $b$ a positive integer, first recall the polynomial identity 
$$\begin{bmatrix} t \\ b \end{bmatrix} = (-1)^b \begin{bmatrix} b-t-1 \\ b \end{bmatrix}.$$
Now, the $(2r-3-j)$-th term is 
\begin{align*}
&(-1)^{2r-3-j} {2r-3 \choose 2r-3-j} \begin{bmatrix} 3r-6-2(2r-3-j) \\ 2r-5 \end{bmatrix} \\
&=(-1)^{-3-j} {2r-3 \choose j} \begin{bmatrix} -r+2j \\ 2r-5 \end{bmatrix} \\
&= (-1)^{-3-j} {2r-3 \choose j} (-1)^{2r-5} \begin{bmatrix} 2r-5+r-2j-1 \\ 2r-5 \end{bmatrix} \\
&= (-1)^{j} {2r-3 \choose j} \begin{bmatrix} 3r-6-2j \\ 2r-5 \end{bmatrix}, 
\end{align*} which is precisely the $j$-th term. Consequently,
$$\sum_{j = 0}^{r-2} (-1)^{j}  {2r-3 \choose j}   \begin{bmatrix}  3r-6-2j \\ 2r-5 \end{bmatrix}=0. $$ For those values of $j$, the binomial polynomial coincides with the binomial coefficient, hence 
$$\sum_{j = 0}^{r-2} (-1)^{j}  {2r-3 \choose j}   {3r-6-2j \choose  2r-5 }=0. $$ Finally, the binomial coefficient 
$${3r-6-2j \choose  2r-5 } $$ is already zero for any $j > r-2$, so that 
$$\sum_{j \ge 0} (-1)^{j}  {2r-3 \choose j}   {3r-6-2j \choose  2r-5 }=0. $$
\end{proof}

Finally, from $\reg(\cap_{i \in [n]} I_i) = n-1$ and Lemma \ref{lem:zero in degree n-2}, it immediately follows that $I_X := \cap_{i \in [n]} I_i$ has a linear graded minimal free resolution.

\section{Proof of Theorem \ref{thm:JL}} \label{section:proof-Thm-JL}

Recalling that $L_c$ is an ideal of $S$ generated by $c$ generic linear forms $\underline{h}=h_1,\dots,h_c$, we begin by recording some basic observations:

\begin{lemma} \label{lem:LJ-first}  
With $J$ any homogeneous ideal, we have:
\begin{enumerate}[(i)]
\item $\reg(J\cap L_c)\leq \reg(J)+1$. 
\item $(J\cap L_c)/JL_c$ has finite length. 
\item If $\reg(J+L_c)<\reg(J)$ then $\reg(J\cap L_c)=\reg(J)$. 
\end{enumerate} 
\end{lemma} 
\begin{proof} 
(i) We have  a short exact sequence:  
\begin{align}
0\to \frac{S}{J\cap L_c} \to \frac{S}{J} \oplus \frac{S}{L_c} \to \frac{S}{J+L_c}\to 0. \tag{$*$}  
\end{align}
Since $L_c$ is generated by generic linear forms, Proposition \ref{prp:almost regular} gives  $\reg (S/(J+L_c))\leq \reg (S/J)$. Furthermore $\reg(S/L_c)=0$. Now the claim follows from the regularity lemma (Proposition \ref{prp:regularity lemma}). 

(ii) Let $h_i, \, i \in [c]$, be generic linear forms that generate the ideal $L_c$, and let $\underline{h}$ denote the sequence $h_1,\dots,h_c$. The sequence $\underline{h}$ is almost regular on $S/J$. Hence the Koszul homology $H_i(\underline{h},S/J)$ has finite length for all $i>0$. In particular $H_1(\underline{h},S/J)=\Tor^S_1(S/J, S/L_c)=(J\cap L_c)/JL_c$ has finite length. 

(iii) This follows from the exact sequence $(*)$ together with the regularity lemma.
\end{proof} 

The next step, even though intuitively non-surprising, is the hardest part behind the proof of Theorem \ref{thm:JL}:
 
\begin{proposition} \label{prp:Lc-Lc+1}  For every non-zero saturated homogeneous ideal $J\subset S$ and every $c\in [r-1]$ we have 
$$\reg(J\cap L_c)\geq \reg(J\cap L_{c+1}).$$
\end{proposition} 
\begin{proof} We discuss first the case $c<r-1$. Suppose, by contradiction that $\reg(J\cap L_c)<\reg(J\cap L_{c+1})$.  Let us write $L_{c+1}=L_c+L_1'$, where $L_1'=(h_{c+1})$ with $h_{c+1}$ a generic linear form. Set 
$$K= (J\cap L_{c}) + (J\cap L_1').$$
Note that we have: 
$$JL_{c+1}=JL_1'+JL_c\subseteq K \subseteq  J\cap L_{c+1}.$$
Since by Lemma \ref{lem:LJ-first}(ii), $(J\cap L_{c+1})/JL_{c+1}$ has finite length, $(J\cap L_{c+1})/K$ also has finite length, and so by Proposition \ref{prp:regularity lemma} $$\reg \left(\frac{S}{J\cap L_{c+1}}\right)\leq \reg\left(\frac{S}{K}\right).$$ 

We have an exact sequence  
\begin{align*}  
0\to \frac{S}{J\cap L_c\cap L_1'} \to \frac{S}{J\cap L_1'}\oplus   \frac{S}{J\cap L_c} \stackrel{f}\to   \frac{S}{K}\to 0
\end{align*} 
 and we further have a canonical projection 
 $$g: \frac{S}{K}\to \frac{S}{J\cap L_{c+1}}.$$ 
 Note that the restriction of $f$ to $S/(J\cap L_1')$ composed with $g$ gives the canonical projection
 $$\pi_{1,c+1}: \frac{S}{J\cap L_1'} \to \frac{S}{J\cap L_{c+1}}.$$
 Now we have
 \begin{align*}
 \reg\left(\frac{S}{J\cap L_c\cap L_1'}\right) \leq \reg\left(\frac{S}{J\cap L_c}\right)+1\leq \reg\left(\frac{S}{J\cap L_{c+1}}\right),
 \end{align*} where the first inequality follows from Lemma \ref{lem:LJ-first}(i) applied to the ideal $J\cap L_c$, and the second inequality is by hypothesis. 
 
 Let $i\in [r]$ and $a\in \NN$ be such that 
 $$\left[H^i_\mf \left(\frac{S}{J\cap L_{c+1}}\right)\right]_a\neq 0 \mbox{ and }  i+a=\reg\left(\frac{S}{J\cap L_{c+1}}\right). $$ Since $J$ is saturated, $J\cap L_{c+1}$ is as well saturated; thus $i>0$.  
 
Consider the maps induced in local cohomology by the above short exact sequence: 
 
 $$\begin{matrix} 
 \left[H^i_\mf\left(\frac{S}{J\cap L_1'}\right)\right]_a \\ 
 \oplus \\
 \left[H^i_\mf\left(\frac{S}{J\cap L_c}\right)\right]_a
 \end{matrix}
  \to \left[H_\mf^i\left(\frac{S}{J\cap L_{c+1}}\right)\right]_a \to \left[H^{i+1}_\mf\left(\frac{S}{J\cap L_c\cap L_1'}\right)\right]_a, $$
 where we have used the fact that  $$H_\mf^i\left(\frac{S}{K}\right)\simeq H_m^i\left(\frac{S}{J\cap L_{c+1}}\right)$$ via the map induced by the projection $g$. Now $$\left[H^i_\mf\left(\frac{S}{J\cap L_c}\right)\right]_a=0,$$ since $$a+i=\reg\left(\frac{S}{J\cap L_{c+1}}\right)>\reg\left(\frac{S}{J\cap L_{c}}\right).$$ Moreover, 
$$\left[H^{i+1}_\mf\left(\frac{S}{J\cap L_c\cap L_1'}\right)\right]_a=0,$$ since 
$$i+1+a=1+\reg\left(\frac{S}{J\cap L_{c+1}}\right)>\reg\left(\frac{S}{J\cap L_c\cap L_1'}\right).$$ 
Therefore we have a surjective map 
\begin{align*} 
\varepsilon: \left[H^i_\mf\left(\frac{S}{J\cap L_1'}\right)\right]_a  \to \left[H_\mf^i\left(\frac{S}{J\cap L_{c+1}}\right)\right]_a  \neq  0,
\end{align*} 
which, as observed above, is induced by the canonical projection $\pi_{1,c+1}$. But $\pi_{1,c+1}$ is the composition of the two canonical projections 
$$ \frac{S}{J\cap L_1'}\stackrel{\pi_{1,c}}\to  \frac{S}{J\cap L_c'}\stackrel{\pi_{c,c+1}}\to   \frac{S}{J\cap L_{c+1}},$$ where $L_c'$ is an ideal generated by $c$ generic linear forms, contained in $L_{c+1}$ and containing $L_1'$. The above composition induces maps 
\begin{align*} 
\left[H^i_\mf\left(\frac{S}{J\cap L_1'}\right)\right]_a  \to \left[H_\mf^i\left(\frac{S}{J\cap L_{c}'}\right)\right]_a  \to \left[H_\mf^i\left(\frac{S}{J\cap L_{c+1}}\right)\right]_a,
\end{align*} whose composition is the above surjective map $\epsilon$. But as we have already seen, the middle term is zero, which is a contradiction. 
   
It remains to discuss the case $c=r-1$. In this special case, since $L_{c+1}=L_r=\mf$ and $J\cap L_{c+1}=J$, we have to prove that $\reg(J\cap L_{r-1})\geq \reg(J)$. As $(J+L_c)/L_c$ is a principal ideal of $S/L_c$, say generated by the class of a polynomial $p$ of degree equal to the initial degree of $J$, we have $J+L_{r-1}=(p)+L_{r-1}$ where $\deg (p)=\min\{ a : J_a\neq 0\}$. Note that $\reg(J+L_{r-1})=\deg (p)$. If $\deg(p)<\reg(J)$ we conclude by Lemma \ref{lem:LJ-first}(iii) that $\reg(J\cap L_{r-1})=\reg(J)$.  Otherwise $\deg(p)=\reg(J)$ (one has that $J$ has a linear resolution) and the short exact sequence $(*)$ gives 
 $$0\to H^0_\mf\left(\frac{S}{J+L_{r-1}}\right)= \frac{S}{J+L_{r-1}} \to H^1_\mf\left(\frac{S}{J\cap L_{r-1}}\right).$$ It follows that 
 $$ \reg \left( \frac{S}{J \cap L_{r-1}} \right) \ge 1 + \reg\left(\frac{S}{J+L_{r-1}}\right) = \reg(J),$$
 that is $\reg(J\cap L_{r-1})\geq \reg(J)+1$.
 \end{proof} 
 
 We can now refine Lemma \ref{lem:LJ-first}:
 
 \begin{proposition} \label{prp:JL-iff}
 Suppose that $J$ is a saturated ideal. Then $\reg(J\cap L_c)=\reg(J)+1$ if and only if $\reg(J+L_c)=\reg(J)$.
 \end{proposition}
 \begin{proof}
 According to Lemma \ref{lem:LJ-first}(i) and Proposition \ref{prp:Lc-Lc+1}, the regularity of $J \cap L_c$ can either be $\reg(J)+1$ or $\reg(J)$. 
 
 By Lemma \ref{lem:LJ-first}(iii) we know that $\reg(J\cap L_c)=\reg(J)+1$ implies $\reg(J+L_c)=\reg(J)$. We argue for the reverse direction. For the sake of a contradiction, suppose $$\reg(J+L_c)=\reg(J)=\reg(J \cap L_c).$$ 
  
 We may assume $\reg(J+L_{c+1})<\reg(J)$; otherwise we simply replace $L_c$ with $L_{c+1}$, noting that $\reg(J\cap L_c)=\reg(J)$ implies $\reg(J\cap L_{c+1})=\reg(J)$ by Proposition \ref{prp:Lc-Lc+1}. Set $M=S/(J+L_c)$. The generic hyperplane section $M/x M$ of $M$ can be realized by taking $x=h_{c+1}$, where $h_{c+1}$ is a generic linear form, i.e. $M/xM = S/(J+L_{c+1})$. Since $x$ is almost regular on $M$, the kernel $0:_M x$ of the multiplication map $M \stackrel{x}{\rightarrow} M$ has finite length; thus Proposition 20.20 in \cite{Eisenbud:CA} gives  $$ \reg(M)=\max\{ \reg(M/xM), \reg(0:_M x)\}.$$ In our setting, by hypothesis we have $\reg(M/xM)<\reg(M)$, so $\reg(M)=\reg(0:_M x).$ 
Now the short exact sequence $(*)$, together with the assumption that $J$ is saturated, implies that $H_\mf^0(M)$ is a submodule of $H_\mf^1(S/(J\cap L_c))$. Since $0:_M x$ is a submodule of $H_\mf^0(M)$, we have
\begin{align*}
\reg(J)-1&=\reg(S/(J\cap L_c)) \ge 1 + \reg(H_\mf^0(M))\\
& \ge 1 + \reg(0:_M x) = 1 + \reg(M)=\reg(J),
\end{align*} which is a contradiction.
 \end{proof}
 
Theorem \ref{thm:JL} now immediately follows from Proposition \ref{prp:JL-iff} and Lemma \ref{lem:LJ-first}(iii), together with the remark that if $c \le \depth(S/J)$, then $L_c$ is generated by an $S/J$-regular sequence and so $\reg(S/J) = \reg(S/(J+L_c))$. 

\section{Proof of Theorem \ref{thm:main}} \label{section:proof-Thm-main}
 
 For $i \in [n]$ we let $I_i$ be generated by $c_i \ge 2$ generic linear forms. We first prove:
 
 \begin{lemma}\label{lem:n-1}
Under the assumptions of Theorem \ref{thm:main} and with $n \ge 2r-3$, we have $\reg(\cap_{i \in [n]} I_i) \le n-1$. 
 \end{lemma}
 \begin{proof}
 For each $i \in [n]$ let $I_i' \subset I_i$ be an ideal generated by two generic linear forms. By Theorem \ref{thm:codim2} and Lemma \ref{lem:LJ-first}(i), we have $\reg(\cap_{i \in [n]} I_i') \le n-1$. If $c_i = 2$ for every $i \in [n]$, we are done. So we may assume $c_n>2$. We may write $I_n = (h_1,h_2,\dots,h_{c_n})$ and $I_n' = (h_1,h_2)$. Set $J = \cap_{i \in [n-1]} I_i'$, $L_2 = I_n'$  and $L_{c_n} = I_n$. Then Proposition \ref{prp:Lc-Lc+1} gives 
$$n-1 \ge \reg (\cap_{i \in [n]} I_i')  = \reg(J \cap L_2) \ge \reg(J \cap L_{c_n}) = \reg((\cap_{i \in [n-1]} I_i') \cap I_n).$$ Continuing in a similar fashion by inductively applying Proposition \ref{prp:Lc-Lc+1} to the ideal $(\cap_{i \in [n-1]} I_i') \cap I_n$, proves the statement. 
\end{proof}
 
Theorem \ref{thm:main} now follows immediately from Lemma \ref{lem:n-1} and the following sub-additivity property of ideals of subspace arrangements:

\begin{lemma}
Let $I_a \subsetneq \mf, \, a \in \mathcal{A}=[\ell]$ and $J_b \subsetneq \mf, \, b \in \mathcal{B}=[m]$ be ideals generated by generic linear forms, where $\ell$ and $m$ are positive integers. Then $$ \reg \left((\cap_{a \in \mathcal{A}} I_a) \cap (\cap_{b \in \mathcal{B}} J_b)\right) \le \reg (\cap_{a \in \mathcal{A}} I_a) + \reg (\cap_{b \in \mathcal{B}} J_b).$$
\end{lemma}
\begin{proof}
For convenience, we set $I_\mathcal{A} = \cap_{a \in \mathcal{A}} I_a$ and $J_\mathcal{B} = \cap_{b \in \mathcal{B}} J_b$. We have inclusions of ideals 
$$\textstyle (\prod_{a \in \mathcal{A}} I_a) (\prod_{b \in \mathcal{B}} J_b) \subset I_\mathcal{A}  J_\mathcal{B} \subset I_\mathcal{A} \cap J_\mathcal{B}.$$ By Proposition 3.4 in \cite{conca2003castelnuovo}, we have 
$$ \textstyle \left[(\prod_{a \in \mathcal{A}} I_a) (\prod_{b \in \mathcal{B}} J_b)\right]^{\sat} = I_\mathcal{A} \cap J_\mathcal{B},$$ and hence also
$$ \textstyle (I_\mathcal{A}  J_\mathcal{B})^{\sat} = I_\mathcal{A} \cap J_\mathcal{B}.$$
Thus the $S$-module 
$$ {\Tor}_1^S\left(\frac{S}{I_\mathcal{A}}, \frac{S}{J_\mathcal{B}}\right) = \frac{I_\mathcal{A} \cap J_\mathcal{B}} {I_\mathcal{A}J_\mathcal{B}},$$ is zero-dimensional, and so \cite{caviglia2007bounds} gives 
$$\reg \left(\frac{S}{I_\mathcal{A}+J_\mathcal{B}} \right)=\reg \left(\frac{S}{I_\mathcal{A}} \otimes_S \frac{S}{J_\mathcal{B}} \right) \le \reg \left(\frac{S}{I_\mathcal{A}}\right) + \reg\left(\frac{S}{J_\mathcal{B}} \right).$$ With this, the short exact sequence 
$$0 \rightarrow \frac{S}{I_\mathcal{A} \cap J_\mathcal{B}} \rightarrow \frac{S}{I_\mathcal{A}} \oplus \frac{S}{J_\mathcal{B}} \rightarrow \frac{S}{I_\mathcal{A} + J_\mathcal{B}} \rightarrow 0$$ and the regularity lemma (Proposition \ref{prp:regularity lemma}) give the assertion. \end{proof}

\bibliographystyle{amsalpha}
\bibliography{Conca-Tsakiris-JEMS-arXiv.bbl}

\end{document}